\documentclass[11pt]{amsart}
\usepackage[margin=1.2in]{geometry}
\usepackage{amsmath,amsthm,amssymb,mathtools,amsfonts,amsopn,amscd}
\usepackage{tikz}
\usepackage{tikz-cd}    
\usepackage{rotating}
\usepackage{commath}    
\usepackage{mathtools}    
\usepackage{graphicx}
\usepackage{etoolbox}
\usepackage{enumitem}
\usepackage{comment}
\usepackage{stmaryrd}
\usepackage{hyperref}
\usepackage{subfiles}
\usepackage{mathrsfs}
\usepackage{calc}
\usepackage[normalem]{ulem}
\usepackage[numbered]{bookmark}     

\usepackage[
backend=biber,
style=alphabetic,
sorting=anyt,
]{biblatex}

\hypersetup{
	colorlinks=true,
	linkcolor=blue,
	urlcolor=black,
	citecolor=red,
	linktoc=page,
}

\theoremstyle{plain}
\newtheorem{thm}{Theorem}[section]
\newtheorem{lem}[thm]{Lemma}

\newtheorem{cor}[thm]{Corollary}

\theoremstyle{definition}
\newtheorem{defn}[thm]{Definition}

\newtheorem{claim}{Claim}

\theoremstyle{remark}
\newtheorem{rem}[thm]{Remark}

\newcommand{\Z}{\mathbb{Z}}
\newcommand{\Q}{\mathbb{Q}}
\newcommand{\F}{\mathbb{F}}
\newcommand{\bbN}{\mathbb{N}}

\newcommand{\idx}{t}

\newcommand{\lrfloor}[1]{\left\lfloor #1 \right\rfloor}
\newcommand{\ang}[1]{\langle #1 \rangle}

\newcommand{\rfrac}[2]{\hbox{\lower+0.3ex\hbox{\scalebox{0.7}[0.67]{$#2$}}\hbox{\kern-0.08em\scalebox{0.97}[1]{$\backslash$}}\lower-0.6ex\hbox{\kern-0.08em \scalebox{0.7}[0.67]{$#1$}}}}
\newcommand{\rdp}{\operatornamewithlimits{\hbox{\lower-0.0ex\hbox{\scalebox{1}[1]{$\prod$}}\lower+0.00ex\hbox{\kern-0.95em \scalebox{1}[1]{$\coprod$}}}}}

\newcommand{\ZZ}{\mathbb{Z}}
\newcommand{\NN}{\mathbb{N}}
\newcommand{\QQ}{\mathbb{Q}}

\newcommand{\FF}{\mathbb{F}}
\newcommand{\GG}{\mathbb{G}}
\newcommand{\oo}{\mathcal{O}}
\newcommand{\xx}{\mathbf{x}}
\newcommand{\yy}{\mathbf{y}}
\newcommand{\ppi}{\boldsymbol{\pi}}
\renewcommand{\tt}{\mathfrak{v}}
\newcommand{\bb}{\mathbf{b}}

\newcommand{\1}{\mathbf{1}}
\newcommand{\GGamma}{\mathbf{\Gamma}}
\newcommand{\rr}{r}

\newcommand{\nfk}{\mathfrak{n}}

\newcommand{\Pfk}{\mathfrak{P}}

\newcommand{\Ocal}{\mathcal{O}}

\newcommand{\End}{\operatorname{End}}

\newcommand{\Gal}{\operatorname{Gal}}

\newcommand{\Frac}{\operatorname{Frac}}

\newcommand{\Tr}{\operatorname{Tr}}

\newcommand{\ord}{\operatorname{ord}}

\newcommand{\ovl}{\overline}

\newcommand{\sbe}{\subseteq}

\let\oldforall\forall
\renewcommand{\forall}{\oldforall \: }
\let\oldexist\exists
\renewcommand{\exists}{\oldexist \: }

\usepackage{stmaryrd}
\SetSymbolFont{stmry}{bold}{U}{stmry}{m}{n}

\newcommand{\Fq}{\FF_q}

\newcommand{\Fqd}{\FF_{q^d}}
\newcommand{\Fqdl}{\FF_{q^{d\l}}}

\newcommand{\T}{\theta}
\newcommand{\Ami}{A_{+,i}}

\newcommand{\vag}{\Gamma_v^{\textnormal{ari}}}

\newcommand{\vgg}{\Gamma_v^{\textnormal{geo}}}

\newcommand{\vtg}{\Gamma_v^{\textnormal{two}}}

\newcommand{\bggs}{g^{\textnormal{geo}}}
\newcommand{\bags}{g^{\textnormal{ari}}}
\newcommand{\ggs}{G_\l^{\textnormal{geo}}}
\newcommand{\ags}{G_\l^{\textnormal{ari}}}

\let\l\ell

\makeatletter

\newcommand{\myToC}{{
		\renewcommand{\contentsname}{}
		\@starttoc{toc}{\contentsname}
}}

\patchcmd{\@tocline}
{\hfil}
{\leaders\hbox{\,.\,}\hfil}

\makeatother

\title{Uniqueness of $v$-adic Gamma Functions in the Gross-Koblitz-Thakur Formulas}
\author{Ting-Wei Chang and Hung-Chun Tsui}
\date{\today}

\thanks{The first author is supported by the National Science and Technology Council grant no. 109-2115-M-007-017-MY5 and 113-2628-M-007-003. The second author is supported by the National Science and Technology Council grant no. 113-2811-M-007-052.}

\begin{document}
	
	\begin{abstract}
		In this paper, we determine all continuous non-vanishing functions satisfying Gross-Koblitz-Thakur formulas in positive characteristic.
	\end{abstract}
	
	\maketitle
	
	\setcounter{tocdepth}{3}
	\myToC
	
	\section{Introduction}
	
	\subsection{Motivation}
	
	We recall that the classical Euler's gamma function is given by
	$$
	\Gamma(z) := \int_{0}^\infty t^{z-1}e^{-t}dt,
	\quad
	{\rm{Re}}(z)>0,
	$$
	which has meromorphic continuation to the complex plane and interpolates the factorials as one has $\Gamma(n+1) = n!$ for all natural numbers $n$.
	The special gamma values are referred to the values of $\Gamma$ at proper fractions, and they are interesting because of having connection with periods. For instance, it is known by Lerch and Chowla-Selberg (see \cite{sc1967epstein}) that periods of elliptic curves with complex multiplication over number fields can be expressed as products of special gamma values.
	A folklore conjecture asserts that all these special gamma values are transcendental numbers.
	To date, the transcendence of $\Gamma(z)$ is known only when the denominator of $z \in \QQ \setminus \ZZ$ is either $2,4$ or $6$.
	For example, one knows that $\Gamma(1/2) = \sqrt{\pi}$ is transcendental {over $\QQ$}.
	The latter two cases mentioned above are consequences of Chudnovsky's result \cite{chudnovsky1984contributions} on the algebraic independence of a non-zero period and the associated quasi-period of a CM elliptic curve defined over $\ovl{\QQ}$.
	
	Let $p$ be a prime number.  Morita \cite{morita1975padic} studied the $p$-adic interpolation problem of the factorials, and then successfully defined the $p$-adic gamma function $\Gamma_p$, which is the unique continuous function on the $p$-adic integers $\ZZ_p$ such that
	$$
	\Gamma_p(n)=(-1)^n\prod_{\substack{1 \leq t \leq n-1 \\ p \kern 0.1em\nmid\kern 0.1em t}} t,
	\quad
	\forall
	n\in\NN.
	$$
	The values $\{ \Gamma_{p}(z) : z\in (\QQ\cap \ZZ_{p})\setminus \ZZ \}$ are called special $p$-adic gamma values. Compared with the (conjectural) situation of Euler's gamma, a peculiar phenomenon in the $p$-adic world occurs: A class of special $p$-adic gamma values turns out to be algebraic numbers when $p>2$. This result is a consequence of the celebrated \textit{Gross-Koblitz formula} in \cite{gk1979gauss}, which is reviewed as follows.
	
	We let $p$ be an odd prime and $q := p^\l$ for some $\l \in\NN$.
	Consider the $(q-1)$-st cyclotomic field $K := \QQ(\mu_{q-1})$.
	We choose a prime $\Pfk$ of $K$ above $p$ and let $\FF_\Pfk$ be its residue field (which is a finite field of $q$ elements).
	We let $\chi$ be the group homomorphism from $\FF_\Pfk^\times$ to $\mu_{q-1}$ in $\ovl{\QQ}_p^\times$ which is the inverse of reduction map (the Teichmüller character), and choose an isomorphism $\psi:\ZZ/p\ZZ \to \mu_p$ as abelian groups.
	For $x \in (q-1)^{-1}\ZZ$ with $0\leq x < 1$, we consider the Gauss sum
	$$
	g_\l(x) := -\sum_{z \in \FF_\Pfk^\times} \chi(z^{-x(q-1)}) \psi(\Tr_{\FF_\Pfk/\F_p}(z)).
	$$
	Let $\varpi_p \in \QQ_p(\psi(1))$ be the unique $(p-1)$-st root of $-p$ for which $\varpi_p\equiv \psi(1)-1\pmod{(\psi(1)-1)^2}$.
	
	\begin{thm}[Gross-Koblitz formula]
		For $x \in (q-1)^{-1}\ZZ$ with $0\leq x < 1$, one has
		\begin{equation}\label{GK formula}
			g_\l(x) = \varpi_p^{(p-1)\sum_{j=0}^{\l-1} \ang{p^jx}} \prod_{j=0}^{\l-1} \Gamma_p\left( \ang{p^jx} \right).
		\end{equation}
		Here, $\ang{y} := y-[y]$ is the fractional part of a non-negative real number $y$.
	\end{thm}
	
	In particular, the Gross-Koblitz formula together with the translation formula of $\Gamma_p$ (see \cite[p. 256]{morita1975padic}) implies that $\Gamma_p (r/(p-1))$ is algebraic over $\QQ$ for all $r \in\ZZ$.
	
	After the above formula was established, Greenberg asked a question concerning the uniqueness of $p$-adic gamma function in Gross-Koblitz formula:
	Is there any other continuous function $H$ from $\ZZ_p$ to $\QQ_p$ such that $\eqref{GK formula}$ remains true when $\Gamma_p$ is replaced by $H$?
	The answer to this problem was provided by Adolphson in \cite{adolphson1983uniqueness} where a complete classification of such $H$ was given.
	Specifically, we consider the function $\varphi:\Z_p\to\Z_p$ defined by
	\begin{equation}    \label{Ado-phi}
		\varphi\left(\sum_{i=0}^\infty x_ip^i\right)
		:= \sum_{i=0}^\infty x_{i+1}p^i
	\end{equation}
	where $0\leq x_i < p$ for all $i$.
	Then he showed (see \cite[pp. 58--59]{adolphson1983uniqueness})
	
	\begin{thm}[Adolphson]    \label{ado result}
		A function $H: \Z_p\to \Q_p$ is continuous non-vanishing satisfying \eqref{GK formula} if and only if
		$$
		H(x) = \Gamma_p(x)\cdot\frac{G(x)}{G\left(-\varphi(-x)\right)}
		$$
		on $\ZZ_p$, where $G:\Z_p\to\Q_p$ is any continuous non-vanishing function.
	\end{thm}
	
	Adolphson's result motivates us to study the positive characteristic counterparts in the function field setting.
	
	\subsection{\texorpdfstring{$v$}{v}-adic Gamma Functions}
	
	We now turn our focus to the function field setting.
	Let $p$ be any prime and $q$ be a power of $p$.
	Let $A:=\Fq[\T]$ be the polynomial ring in the variable $\T$ over a finite field $\Fq$ of $q$ elements,  and $k := \Fq(\T)$ be its field of fractions.
	Let $A_+$ be the set of all monic polynomials in $A$ and $A_{+,i}$ (for $i\geq0$) be its subset consisting of all monic polynomials of degree $i$.
	We fix an irreducible $v\in A_{+,d}$ and let $k_v$ be the completion of $k$ at $v$ with ring of integers $A_v$.
	We let $\ovl{k}$ be the algebraic closure of $k$ contained in a fixed algebraic closure $\ovl{k}_v$ of $k_v$.
	For an element $x\in A_v$, put
	$$
	x^\flat := 
	\begin{cases}
		x, & \text{if } x \in A_v^\times, \\
		1, & \text{if } x\in v A_v.
	\end{cases}
	$$
	
	\begin{defn}    \label{v-adic-gammas-definition}
		(1) (\textit{$v$-adic arithmetic gamma function}, \cite[Appendix]{goss1980modular}): For $y = \sum_{i=0}^{\infty} y_iq^i \in \ZZ_p$ with $0\leq y_i < q$ for all $i$, define $\vag: \ZZ_p \to A_v^\times$ by
		$$
		\vag (y+1) := \prod_{i=0}^{\infty} \left( -\prod_{a\in \Ami} a^\flat \right)^{y_i}.
		$$
		(2) (\textit{$v$-adic geometric gamma function}, \cite[Section 5]{thakur1991gamma}): Define $\vgg: A_v \to A_v^\times$ by
		$$
		\vgg(x) := \frac{1}{x^\flat} \prod_{i=0}^{\infty} \left( \prod_{a\in \Ami} \frac{a^\flat}{(x+a)^\flat} \right).
		$$
		(3) (\textit{$v$-adic two-variable gamma function}, \cite[Subsection 9.9]{goss1996basic}): For $(x,y) \in A_v \times \ZZ_p$ with $y$ as in the arithmetic case, define $\vtg: A_v \times \ZZ_p \to A_v^\times$ by
		$$
		\vtg(x,y+1) := \frac{1}{x^\flat} \prod_{i=0}^\infty \left( \prod_{a\in\Ami} \frac{a^\flat}{(x+a)^\flat} \right)^{y_i}.
		$$
	\end{defn}
	
	The arithmetic case is the Morita-style $v$-adic interpolation of Carlitz factorials \cite{carlitz1935oncertain} $\Gamma^{\textnormal{ari}}: \NN \to A$ given by
	$$
	\Gamma^{\textnormal{ari}} \left(1 + \sum_{i=0}^{n} y_iq^i \right)
	:= \prod_{i=0}^n \left( \prod_{a\in \Ami} a \right)^{y_i}
	\quad
	\text{where}
	\quad
	0 \leq y_i < q.
	$$
	The above can be viewed as an analog of $\Gamma(y+1) = y!$ for $y \in \NN$ in view of W. Sinnott's result that (see \cite[Theorem 9.1.1]{goss1996basic})
	$$
	\ord_f (\Gamma^{\textnormal{ari}} (y+1))
	= \sum_{e \geq 1} \lrfloor{\frac{y}{q^{e\deg f}}} 
	$$
	for any monic prime $f$.
	On the other hand, the geometric case arises from the Morita-style $v$-adic interpolation of the function $\Gamma^{\textnormal{geo}}: A \to k \cup \{\infty\}$ given by
	$$
	\Gamma^{\textnormal{geo}}(x)
	:= \frac{1}{x} \prod_{a\in A_+} \frac{a}{x+a}
	= x^{-1}  \prod_{a\in A_+} \left( 1+\frac{x}{a} \right)^{-1}.
	$$
	Note that $\Gamma^{\textnormal{geo}}$ has simple poles at $-A_+ \cup \{0\} := \{-a \mid a\in A_+\} \cup \{0\}$, analogous to the case that Euler's $\Gamma$-function has simple poles at $-\NN \cup \{0\} := \{-n \mid n\in \NN\} \cup \{0\}$.
	Finally, for the two-variable case, one sees that
	$$
	\vtg\left(x,1-\frac{1}{q-1}\right)
	= \vgg(x).
	$$
	Note that all three gamma functions in Definition \ref{v-adic-gammas-definition} are continuous and non-vanishing as they are defined via the interpolations of continuous functions and take values in the units of $A_v$.
	
	\subsection{Gross-Koblitz-Thakur Formulas for \texorpdfstring{$v$}{v}-adic Gamma Functions}
	
	There are ana\-logs of Gross-Koblitz formulas for these $v$-adic gamma functions by the work of Thakur \cite{thakur1988gauss} and the first author \cite{chang2025geometric}.
	Before stating the precise formulas we prepare some necessary background (see \cite[Chapter 3]{goss1996basic} for more details).
	Let $K$ be a field together with an $\Fq$-algebra homomorphism $\iota:A \to K$.
	Let $\tau \in \End_{\Fq}(\GG_a)$ be the $q$-th power Frobenius morphism on the additive group $\GG_a$ over $K$.
	We let $K\{\tau\}$ be the twisted polynomial ring in $\tau$ over $K$ with usual addition but the multiplication is given by $\tau x = x^q \tau$ for all $x \in K$.
	Recall that the Carlitz module over $K$ is the $\Fq$-algebra homomorphism $C: A \to K\{\tau\}$ defined by $C_\T := \iota(\T) + \tau$.
	It induces a new $A$-module structure $C(K)$ on $K$ given by $a \cdot x := C_a(x)$ for all $a \in A$ and $x \in K$.
	In the case $K = \ovl{k}$ (with the natural structure map), we let $\Lambda_a := C(\ovl{k})[a]$ be the $a$-torsion points of $C(\ovl{k})$ for any $a \in A$.
	Then it is a fact that $\Lambda_a$ is isomorphic to $A/a$ as $A$-modules.
	
	To state the analogs of Gross-Koblitz formulas for three $v$-adic gamma functions, we define the notions of fractional part on the rational number field $\QQ$ and function field $k$.
	Note for $y \in \QQ$, its fractional part $\ang{y}$ can be defined as the unique element in $\QQ$ such that $0 \leq \ang{y} <1$ and $y \equiv \ang{y} \pmod{\ZZ}$.
	On the other hand, one recalls the normalized $\infty$-adic absolute value on $k$ is given by $|0|_\infty := 0$ and $|f/g|_\infty := q^{\deg f - \deg g}$ for any $f,g \in A \setminus\{0\}$.
	(In what follows, we will omit the subscript $|\cdot|_\infty$ and denote it as $|\cdot|$ for convenience.)
	Thus for $x \in k$, we define its fractional part $\ang{x}$ as the unique element in $k$ such that $0 \leq |\ang{x}| < 1$ and $x \equiv \ang{x} \pmod{A}$.
	
	First, we consider the arithmetic case \cite{thakur1988gauss}.
	Let $\chi: A/v \to \Fqd \sbe \ovl{k}$ be the Teichmüller character, and choose an $A$-module isomorphism $\psi: A/v \to \Lambda_v$.
	Then Thakur defined the \textit{arithmetic Gauss sum} to be
	$$
	\bags
	:= -\sum_{z\in (A/v)^\times} \chi(z^{-1}) \psi(z) \in k(\Lambda_v)\Fqd.
	$$
	(The same element is denoted as $g_0$ in the original paper.)
	More generally, take any natural number $\l$.
	For $y \in (q^{d\l}-1)^{-1}\ZZ$ with $0 \leq y < 1$, we write
	$$
	y = \sum_{s=0}^{d\l-1} \frac{y_s q^s}{q^{d\l}-1}
	\quad
	(0 \leq y_s < q \text{ for all } s).
	$$
	We let $\tau_q \in \Gal(k\Fqd/k)$ be the $q$-th power Frobenius map on the constant field, and extend it canonically to $\Gal(k(\Lambda_v)\Fqd/k) \simeq \Gal(k(\Lambda_v)/k) \times \Gal(k\Fqd/k)$.
	Then we consider the integral group ring element $\sum_{s=0}^{d\l-1} y_s \tau_q^s \in \ZZ[\Gal(k(\Lambda_v)\Fqd/k)]$ acting on the arithmetic Gauss sum $\bags$, and define
	$$
	\ags(y) := (-1)^{\l(d-1)} \prod_{s=0}^{d\l-1} (\bags)^{y_s \tau_q^s}.
	$$
	Let $\varpi_v \in k_v(\psi(1))$ be the unique $(q^d-1)$-st root of $-v$ for which $\varpi_v \equiv -\psi(1) \pmod{\psi(1)^2}$.
	Then under this setting, Thakur proved the following analog of Gross-Koblitz formula for $v$-adic arithmetic gamma function (see \cite[Theorem VI]{thakur1988gauss} and \cite[Theorem 4.8]{thakur1991gamma}).
	
	\begin{thm}[Gross-Koblitz-Thakur formula, arithmetic case]
		For $y \in (q^{d\l} -1)^{-1}\ZZ$ with $0 \leq y < 1$, one has
		\begin{equation} \label{GKT formula ari}
			G_\l^{\textnormal{ari}}(y)
			= \varpi_v^{(q^d-1)\sum_{j=0}^{\l-1} \ang{q^{dj}y}} \prod_{j=0}^{\l-1} \vag\left(\ang{q^{dj}y}\right). 
		\end{equation}
	\end{thm}
	
	Next, we consider the geometric and two-variable cases \cite{chang2025geometric}.
	Fix a positive integer $\l$ and let $\nfk := v^\l$.
	Consider the Carlitz $(\nfk-1)$-st cyclotomic function field $K := k(\Lambda_{\nfk-1})$.
	We let $\Pfk$ be the prime in $K$ above $v$ corresponding to the inclusions $K \sbe \ovl{k} \sbe \ovl{k}_v$, and let $\FF_\Pfk$ be the residue field (which is a finite field of $q^{d\l}$ elements).
	Then via the natural structure map on $\FF_\Pfk$, it obtains a new $A$-module structure $C(\FF_\Pfk)$ via the Carlitz module.
	We let $\omega: C(\FF_\Pfk) \to \Lambda_{\nfk-1}$ be the $A$-module isomorphism which is the inverse of reduction map, and let $\psi: \FF_\Pfk \to \Fqdl \sbe \ovl{k}$ be the usual Teichmüller embedding.
	Then for $x \in (\nfk-1)^{-1}A$ with $|x| < 1$, we define the \textit{geometric Gauss sum} to be
	$$
	\bggs_x
	:= 1 + \sum_{z \in \FF_\Pfk^\times} \omega\left(C_{x(\nfk-1)}(z^{-1})\right)\psi(z) \in K\Fqdl.
	$$
	More generally, we let $\tau_q \in \Gal(k\Fqdl/k)$ be the $q$-th power Frobenius map on the constant field, and extend it canonically to $\Gal(K\Fqdl/k) \simeq \Gal(K/k) \times \Gal(k\Fqdl/k)$.
	For $y \in (q^{d\l}-1)^{-1}\ZZ$ with $0 \leq y < 1$ written as in the arithmetic case, we consider the integral group ring element $\sum_{s=0}^{d\l-1} y_s \tau_q^s \in \ZZ[\Gal(K\Fqdl/k)]$ acting on the geometric Gauss sum $\bggs_x$, and define
	$$
	\ggs(x,y) := \prod_{s=0}^{d\l-1} (\bggs_x)^{y_s \tau_q^s}.
	$$
	In particular, put
	$$
	\ggs(x)
	:= \ggs \left(x , \frac{1}{q-1} \right)
	= \prod_{s=0}^{d\l-1} (\bggs_x)^{\tau_q^s}.
	$$
	Then under this setting, we have the following analogs of Gross-Koblitz-Thakur formulas for $v$-adic geometric and two-variable gamma functions \cite[Theorems 5.1.3 and 5.1.4]{chang2025geometric}.
	
	\begin{thm}[Gross-Koblitz-Thakur formulas, geometric and two-variable cases] 
		Suppose $x \in (\nfk-1)^{-1} A$ and $y \in (q^{d\l}-1)^{-1} \ZZ$ with $|x|<1$ and $0 \leq y <1$, write
		$$
		x = \sum_{j=0}^{\l-1} \frac{x_j v^j}{v^\l-1}
		\quad
		(\deg x_j < d \text{ for all } j),
		\quad
		y = \sum_{s=0}^{d\l-1} \frac{y_s q^s}{q^{d\l}-1}
		\quad
		(0 \leq y_s < q \text{ for all } s).
		$$
		Then we have
		\begin{equation}   \label{GKT formula geo}
			\ggs(x)
			= \prod_{j=0}^{\l-1} \frac{\delta_{x,j}}{\ang{v^jx}^\flat}
			\cdot
			\prod_{j=0}^{\l-1} \vgg \left( \ang{v^jx} \right)^{-1}
		\end{equation}
		and
		\begin{equation}    \label{GKT formula two}
			\ggs (x,y)
			= \left( \prod_{j=0}^{\l-1}
			\frac{\ang{v^jx}^\flat}{\delta_{x,j}^{q-1-y_{dj+\deg x_j}}} \right)
			\ggs(x)^{q-1}
			\prod_{j=0}^{\l-1} \vtg \left( \ang{v^jx}, \ang{q^{dj}y} \right)
		\end{equation}
		where $\delta_{x,j} := v\ang{v^{\l-j-1}x}$ if $x_j \in A_+$ and $1$ otherwise.
	\end{thm}
	
	\begin{rem}
		Note that the translation formulas of these $v$-adic gamma functions imply that
		$$
		\frac{\vag(y+r)}{\vag(y)}, 
		\frac{\vgg(x+a)}{\vgg(x)},
		\frac{\vtg(x+a,y+r)}{\vtg(x,y)}
		\in k^\times,
		\quad
		\forall r \in \ZZ, a \in A.
		$$
		Hence, similar to the classical case, the Gross-Koblitz-Thakur formulas \eqref{GKT formula ari},\eqref{GKT formula geo},\eqref{GKT formula two} imply that $\vag(r/(q^d-1)), \vgg(a/(v-1)), \vtg(a/(v-1),r/(q^d-1))$ are algebraic over $k$ for all $r\in \ZZ$ and $a \in A$.
		We further mention that in the paper~\cite{cwy2024vadic}, Chang-Wei-Yu proved the transcendence of those special $v$-adic arithmetic gamma values which are not from the Gross-Koblitz-Thakur formula.
		That is, for $a/b \in \QQ^{\times}\cap \ZZ_{p}$ with $\gcd(a,b)=1$, we have $\vag (a/b)$ is transcendental over $k$ if and only if $b \nmid q^{d}-1$.
	\end{rem}
	
	\subsection{Main Theorem}
	
	In this paper, we consider the analogous question of Greenberg for these three $v$-adic gamma functions:
	For each case, is there any other non-vanishing continuous $k_v$-valued function $H$ defined on the same domain as the respective gamma function, such that the corresponding Gross-Koblitz-Thakur formula (\eqref{GKT formula ari},\eqref{GKT formula geo},\eqref{GKT formula two}) remains true when the gamma is replaced by $H$?
	For this problem, we give a precise description of all such functions parallel to the classical result.
	
	To treat these three cases at once, we consider a more general setting.
	Let $\Ocal$ be a direct product of finitely many rings $\Ocal_\idx$ (indexed by $\idx$), where each $\Ocal_\idx$ is either $\ZZ$ or $A$.
	We let $\ovl{\Ocal}_\idx$ be the completion $\ZZ_{p_\idx}$ or $A_{v_\idx}$ (depending on the choice of $\Ocal_\idx$) at a fixed finite prime $p_\idx$ or $v_\idx$, and let $\ovl{\Ocal}$ be the direct product of $\ovl{\Ocal}_\idx$.
	
	For each $\idx$, we fix a power of the given uniformizer $\pi_\idx \in \ovl{\Ocal}_\idx$ and put $\ppi := (\pi_\idx)_\idx \in \ovl{\Ocal}$.
	Let $\1 := (1,\ldots,1) \in\ovl{\Ocal}$ be the identity element.
	For $\xx = (\xx_\idx)_\idx \in \ovl{\Ocal}$, we write
	$$
	\xx_\idx
	:= \sum_{i=0}^\infty x_{\idx,i}\pi_\idx^i \in \ovl{\Ocal}_\idx
	$$
	in $\pi_\idx$-adic expansion.
	This means that $x_{\idx,i}$ is an integer with $0 \leq x_{\idx,i} < \pi_\idx$ when $\Ocal_\idx = \ZZ$ or a polynomial in $A$ with $|x_{\idx,i}| < |\pi_\idx|$ (i.e., $\deg x_{\idx,i} < \deg \pi_\idx$) when $\Ocal_\idx = A$.
	Define $\varphi : \ovl{\Ocal} \to \ovl{\Ocal}$ by $\varphi(\xx) = (\varphi_\idx(\xx_\idx))_\idx$
	where each $\varphi_\idx: \ovl{\Ocal}_\idx \to \ovl{\Ocal}_\idx$ is given by (cf. \eqref{Ado-phi})
	$$
	\varphi_\idx(\xx_\idx)
	:= \sum_{i=0}^\infty x_{\idx,i+1}\pi_\idx^i.
	$$
	We also define the fractional part of elements in the product of fraction fields of $\oo_t$, denoted as $\prod \Frac(\Ocal_\idx)$, by taking the fractional parts componentwise.
	That is, define $\ang{\xx} := (\ang{\xx_\idx})_\idx$ for any $\xx \in \prod \Frac(\Ocal_\idx)$.
	
	Our main result in this paper is to study and answer analogous problems of Greenberg in the setting as general as possible.
	For any $n \in \NN$ and $\xx = (\xx_\idx)_\idx \in (\ppi^n - \1)^{-1} \Ocal$, we set up the condition:
	\begin{equation}    \label{con}
		\begin{cases}
			0 \leq \xx_t < 1,  &\text{if } \Ocal_\idx = \ZZ, \\
			|\xx_t| < 1,  &\text{if } \Ocal_\idx = A.
		\end{cases}
	\end{equation}
	The following  is our main theorem, which  recovers Theorem \ref{ado result}.
	\begin{thm}   \label{main result}
		Given any complete non-Archimedean field $K$, let $\GGamma: \ovl{\Ocal} \to K$ be a continuous non-vanishing function.
		Then $H: \ovl{\Ocal} \to K$ is a continuous non-vanishing function satisfying for all $n\in\NN$ and $\xx \in (\ppi^n-\1)^{-1} \Ocal$ with condition \eqref{con},
		\begin{equation}   \label{condition-in-main-result}
			\prod_{j=0}^{n-1} \GGamma\left(\ang{\ppi^j \xx}\right)
			= \prod_{j=0}^{n-1} H\left(\ang{\ppi^j \xx}\right)
		\end{equation}
		if and only if
		$$
		H(\xx) = \GGamma(\xx) \cdot \frac{G(\xx)}{G\left(-\varphi(-\xx)\right)}
		$$
		on $\ovl{\Ocal}$, where $G: \ovl{\Ocal} \to K$ is any continuous non-vanishing function.
	\end{thm}
	
	As corollaries, we obtain the complete characterizations of all continuous non-vanishing functions satisfying Gross-Koblitz-Thakur formulas for the arithmetic, geometric and two-variable gamma functions, respectively.
	Specifically, for each $\bullet = \text{ari, geo, two}$, we choose $(\ovl{\Ocal}{}^\bullet,\ppi^\bullet)$ to be $(\ZZ_p,q^d), (A_v,v), (A_v \times \ZZ_p, (v,q^d))$, respectively.
	We denote $\varphi^\bullet: \ovl{\Ocal}{}^\bullet \to \ovl{\Ocal}{}^\bullet$ to be the corresponding $\varphi$ function.
	Then we have the following consequence.
	
	\begin{cor}
		For any $\bullet \in \{ \textnormal{ari, geo, two}\}$, a function $H: \ovl{\Ocal}{}^\bullet \to k_v$ is continuous non-vanishing and satisfies the corresponding Gross-Koblitz-Thakur formula (\eqref{GKT formula ari}, \eqref{GKT formula geo}, \eqref{GKT formula two}) if and only if
		$$
		H(\xx) = \Gamma^\bullet_v(\xx) \cdot \frac{G(\xx)}{G\left(-\varphi^\bullet(-\xx)\right)}
		$$
		on $\ovl{\Ocal}{}^\bullet$, where $G: \ovl{\Ocal}{}^\bullet \to k_v$ is any continuous non-vanishing function.
	\end{cor}
	
	\subsection{Strategy of Proof}
	
	Our proof is inspired by the original paper \cite{adolphson1983uniqueness} (in particular, the Remark 3 and the theorem after it).
	First, it is observed in Section \ref{if} that \eqref{condition-in-main-result} can be rewritten as (for the same conditions on $\xx$)
	$$
	\prod_{j=0}^{n-1} \GGamma\left(-\varphi^{(j)}(-\xx)\right)
	= \prod_{j=0}^{n-1} H\left(-\varphi^{(j)}(-\xx)\right).
	$$
	Such expression will immediately imply the ``if'' direction.
	
	For the ``only if'' part, we set $F(\xx) := H(\xx)/\GGamma(\xx)$.
	Then what we want to show is that $F$ must be of the form (Theorem \ref{thm-only-if} and \cite[Theorem 1]{adolphson1983uniqueness})
	\[F(\xx)=\frac{G(\xx)}{G\left(-\varphi(-\xx)\right)}\]
	on $\ovl{\Ocal}$ for some continuous non-vanishing function $G: \ovl{\Ocal} \to K$.
	To prove this, we consider $\tilde{F}(\xx) := F(-\xx)$ (to eliminate the minus signs; Theorem \ref{change of variables} and \cite[Theorem 2]{adolphson1983uniqueness}) and construct two sequences of continuous functions $\alpha_n(\xx)$ and $\beta_n(\xx)$ (see \eqref{alpha_n} and \eqref{beta_n}).
	Then it is shown in \eqref{FABG} that
	$$
	\tilde{F}\left(\varphi^{(n-1)}\left(\alpha_n(\xx)\right)\right)=A_n(\xx)\cdot B_n(\xx)\cdot \frac{\tilde{G}_n(\xx)}{\tilde{G}_n\left(\varphi(\xx)\right)}
	$$
	for some (explicit in terms of $\tilde{F},\alpha_n,\beta_n$ and $\varphi$) functions $A_n,B_n,\tilde{G}_n: \ovl{\Ocal} \to K$.
	With some special properties of $\tilde{F},\alpha_n$ and $\beta_n$ (see \eqref{new condition}, \eqref{properties of alpha beta}, etc.), we prove that the left-hand side converges to $\tilde{F}(\xx)$ and the right-hand side converges to $\tilde{G}(\xx)/\tilde{G}(\varphi(\xx))$ for some continuous non-vanishing function $\tilde{G}:\ovl{\Ocal} \to K$ (Lemma \ref{all limits}).
	Then the desired function $G(\xx)$ turns out to be $\tilde{G}(-\xx)$, and the result follows.
	
	\section{Proof of the Theorem \ref{main result}}
	
	\subsection{The ``if'' Direction}    \label{if}
	
	In this section, we give a proof of “$\Leftarrow$” of Theorem \ref{main result}. The primary key of our proof is to reformulate \eqref{condition-in-main-result} as follows.
	Recall that $\varphi : \ovl{\Ocal} \to \ovl{\Ocal}$ is defined by $\varphi(\xx) = (\varphi_\idx(\xx_\idx))_\idx$
	where each $\varphi_\idx: \ovl{\Ocal}_\idx \to \ovl{\Ocal}_\idx$ is given by
	$$
	\varphi_\idx(\xx_\idx)
	:= \sum_{i=0}^\infty x_{\idx,i+1}\pi_\idx^i
	$$ 
	for any
	$$
	\xx_\idx
	:= \sum_{i=0}^\infty x_{\idx,i}\pi_\idx^i \in \ovl{\Ocal}_\idx.
	$$
	For $\xx,\yy\in \ovl{\Ocal}$, we denote by
	$$
	\xx\equiv\yy \pmod{\ppi^n}
	$$
	if 
	$$
	\xx_\idx \equiv \yy_\idx \pmod{\pi_\idx^n}
	$$
	for all $\idx$.
	We note that for any $\xx,\yy\in \ovl{\Ocal}$,
	\begin{equation}    \label{uni conti of phi}
		\xx\equiv \yy\pmod {\ppi^{n}}\implies \varphi(\xx)\equiv\varphi(\yy)\pmod {\ppi^{n-1}}.
	\end{equation}
	Furthermore, if $\xx = (\xx_\idx)_\idx\in(\ppi^{n}-\1)^{-1}\oo$ satisfying the condition \eqref{con}, then we have
	$$
	\{\langle \ppi^{j}\xx \rangle:0\leq j\leq n-1\}
	= \{-\varphi^{(j)}(-\xx):0\leq j\leq n-1\}.
	$$
	The identity above follows from expressing each $\xx_\idx = m_\idx/(\pi_\idx^n-1)$ where $0 \leq |m_\idx| < |\pi_\idx^n-1|$
	(this  means that $m_\idx$ is a positive integer with $0 \leq m_\idx < \pi_\idx^n-1$ when $\Ocal_\idx = \ZZ$ or a polynomial in $A$ with $\deg m_\idx < \deg(\pi_\idx^n-1) = \deg \pi_\idx^n$ when $\Ocal_\idx = A$), and observing that the element $\ang{\pi_\idx \xx_\idx}$ (resp. $-\varphi_\idx(-{\xx_\idx})$) in each component corresponds to shifting the $\pi_\idx$-adic digits of $m_\idx$ one step right (resp. left).
	So when $j$ runs through $0$ to $n-1$, two sets are in fact identical.
	To conclude, the formula \eqref{condition-in-main-result} can be rewritten as follows:
	For all $n\in\bbN$ and $\xx \in (\ppi^n-\1)^{-1} \Ocal$ satisfying the condition \eqref{con},
	\begin{equation}\label{reformulation of GKT}
		\prod_{j=0}^{n-1} \GGamma\left(-\varphi^{(j)}(-{\xx})\right)
		= \prod_{j=0}^{n-1} H\left(-\varphi^{(j)}(-{\xx})\right).
	\end{equation}
	
	To complete the proof of “$\Leftarrow$” of Theorem~\ref{main result}, suppose that we are given any continuous non-vanishing function $G:\ovl{\Ocal} \to K$. We claim that the following function
	$$
	H(\xx):=\GGamma(\xx)\cdot\frac{G(\xx)}{G\left(-\varphi(-\xx)\right)}: \ovl{\Ocal}\rightarrow K
	$$
	is continuous, non-vanishing, and satisfies the identity \eqref{reformulation of GKT}, whence the proof would be completed.
	To prove the claim above, we first note that the continuous and non-vanishing properties are clear from the definition of $H$. To show the formula \eqref{reformulation of GKT}, we further mention that for $\xx \in(\ppi^{n}-\1)^{-1}\oo$ with condition \eqref{con}, we have
	\begin{equation}    \label{product-of-G}
		\prod_{j=0}^{n-1}\frac{G\left(-\varphi^{(j)}\left(-{\xx}\right)\right)}{G\left(-\varphi^{(j+1)}\left(-{\xx}\right)\right)}=1
	\end{equation}
	as $-\varphi^{(n)} (-\xx) = \xx$. Hence, the desired identity is satisfied.
	
	\subsection{The ``only if'' Direction}
	
	We now proceed to prove the only if part of Theorem \ref{main result}.
	For any $n \in \NN$ and $\xx = (\xx_\idx)_\idx \in \ovl{\Ocal}$, we observe that $\xx  \in (\ppi^n-\1)^{-1} \Ocal$ and $\xx$ satisfies \eqref{con} if and only if $\varphi^{(n)}(-\xx) = -\xx$ and $|\xx_t| < 1$ for all $t$.
	Let $F(\xx):={H(\xx)}/{\GGamma(\xx)}$, then using the reformulation \eqref{reformulation of GKT}, we need to show the following.
	
	\begin{thm}    \label{thm-only-if}
		Let $F: \ovl{\Ocal} \to K$ be a continuous non-vanishing function. Suppose that for all $n\in\bbN$ and $\xx=(\xx_\idx)_\idx\in \ovl{\Ocal}$ with $|\xx_\idx|<1$ for all $\idx$,
		$$
		\varphi^{(n)}(-\xx)=-\xx \implies \prod_{j=0}^{n-1}F\left(-\varphi^{(j)}(-\xx)\right)=1.
		$$
		Then there exists a continuous non-vanishing function $G:\ovl{\Ocal} \to K$ such that $$F(\xx)=\frac{G(\xx)}{G(-\varphi(-\xx))}$$
		for all $\xx\in \ovl{\Ocal}$.
	\end{thm}
	
	By replacing $F(\xx),G(\xx)$ with $F(-\xx),G(-\xx)$ respectively, Theorem \ref{thm-only-if} is equivalent to the following.
	
	\begin{thm}    \label{change of variables}
		Let $F:\ovl{\Ocal} \to K$ be a continuous non-vanishing function.
		Suppose that for all $n\in\bbN$ and $\xx=(\xx_\idx)_\idx\in\ovl{\Ocal}$ with $|\xx_\idx|<1$ for all $\idx$,
		\begin{equation}\label{new condition}
			\varphi^{(n)}(\xx)=\xx \implies \prod_{j=0}^{n-1}F\left(\varphi^{(j)}(\xx)\right)=1.
		\end{equation}
		Then there exists a continuous non-vanishing function $G:\ovl{\Ocal} \to K$ such that $$F(\xx)=\frac{G(\xx)}{G(\varphi(\xx))}$$
		for all $\xx\in\ovl{\Ocal}$.
	\end{thm}
	
	\begin{proof}
		Let $\bb:=(b_\idx)_\idx\in\oo$ such that for each $\idx$,
		\begin{equation}    \label{bt}
			\begin{cases}
				0 \leq b_t < \pi_t-1,  &\text{if } \Ocal_\idx = \ZZ, \\
				|b_t| < |\pi_t-1|,  &\text{if } \Ocal_\idx = A.
			\end{cases}
		\end{equation}
		Recall for any $\xx = (\xx_\idx)_\idx \in \ovl{\Ocal}$, we write
		\[\xx_\idx := \sum_{i=0}^\infty x_{\idx,i}\pi_\idx^i \in \ovl{\Ocal}_\idx
		\] in $\pi_\idx$-adic expansion.
		We set
		$$
		\xx_{(i)}:=(x_{\idx,i})_\idx\in\oo
		\quad
		\text{so that}
		\quad
		\xx = \sum_{i=0}^{\infty} \xx_{(i)}\ppi^i.
		$$
		Define $\alpha_{n},\beta_n:\ovl{\Ocal} \to \ovl{\Ocal}$ by
		\begin{equation}\label{alpha_n}
			\alpha_{n}(\xx)=\frac{1}{\1-\ppi^{2n-1}}\left(\bb+\bb\ppi+\cdots+\bb \ppi^{n-2}+\ppi^{n-1}\left(\sum_{i=0}^{n-1}\xx_{(i)}\ppi^{i}\right)\right),
		\end{equation}
		and
		\begin{equation}\label{beta_n}
			\beta_{n}(\xx)=\frac{1}{\1-\ppi^{2n-2}}\left(\bb+\bb\ppi+\cdots+\bb \ppi^{n-2}+\ppi^{n-1}\left(\sum_{i=0}^{n-2}\xx_{(i)}\ppi^{i}\right)\right).
		\end{equation}
		We remark that $\alpha_n,\beta_n$ are locally constant and hence continuous on $\ovl{\Ocal}$. Moreover, note that for any $\xx\in\ovl{\Ocal}$, we have
		\[|\alpha_{n}(\xx)_\idx|<1
		\quad
		\text{and}
		\quad
		|\beta_{n}(\xx)_\idx|<1
		\]
		for all $t$ by our choice of $b_\idx$ \eqref{bt}, and 
		\begin{equation}\label{properties of alpha beta}
			\varphi^{(2n-1)}\left(\alpha_n(\xx)\right)=\alpha_n(\xx)
			\quad
			\text{and}
			\quad
			\varphi^{(2n-2)}\left(\beta_n\left(\varphi(\xx)\right)\right)=\beta_n\left(\varphi(\xx)\right).
		\end{equation}
		Hence by \eqref{new condition}, we obtain
		\[\prod_{i=0}^{2n-2}F\left(\varphi^{(i)}\left(\alpha_n(\xx)\right)\right)=1
		\quad\text{and}
		\quad
		\prod_{i=0}^{2n-3}F\left(\varphi^{(i)}\left(\beta_n\left(\varphi(\xx)\right)\right)\right)=1.\]
		By equating these two expressions, we have
		\[F\left(\varphi^{(n-1)}\left(\alpha_n(\xx)\right)\right)=\frac{\prod_{i=0}^{2n-3}F\left(\varphi^{(i)}\left(\beta_n\left(\varphi(\xx)\right)\right)\right)}{\prod_{i=0}^{n-2}F\left(\varphi^{(i)}\left(\alpha_n(\xx)\right)\right)\prod_{i=n}^{2n-2}F\left(\varphi^{(i)}\left(\alpha_n(\xx)\right)\right)}.\]
		By multiplying both the numerator and the denominator of the the right-hand side by $\prod_{i=0}^{n-2}F\left(\varphi^{(i)}\left(\beta_n(\xx)\right)\right)$, we obtain
		\begin{equation}   \label{FABG}
			F\left(\varphi^{(n-1)}\left(\alpha_n(\xx)\right)\right)=A_n(\xx)\cdot B_n(\xx)\cdot \frac{G_n(\xx)}{G_n\left(\varphi(\xx)\right)}
		\end{equation}
		where 
		\begin{equation}\label{An Bn} A_n(\xx):=\frac{\prod_{i=0}^{n-2}F\left(\varphi^{(i)}\left(\beta_n(\xx)\right)\right)}{\prod_{i=0}^{n-2}F\left(\varphi^{(i)}\left(\alpha_n(\xx)\right)\right)},
			\quad
			B_n(\xx):=\frac{\prod_{i=n-1}^{2n-3}F\left(\varphi^{(i)}\left(\beta_n\left(\varphi(\xx)\right)\right)\right)}{\prod_{i=n}^{2n-2}F\left(\varphi^{(i)}\left(\alpha_n(\xx)\right)\right)} 
		\end{equation}
		and
		\begin{equation}\label{def of G_n}
			G_n(\xx):=\left[\prod_{i=0}^{n-2}F\left(\varphi^{(i)}\left(\beta_n(\xx)\right)\right)\right]^{-1}.
		\end{equation}
		
		The result now follows from the following lemma, which will be proved in the next section.
		
		\begin{lem}\label{all limits}
			The following statements hold.
			\begin{enumerate}
				\item $\lim\limits_{n\to\infty}F\left(\varphi^{(n-1)}\left(\alpha_n(\xx)\right)\right)=F(\xx).$
				
				\item $\lim\limits_{n\to\infty}A_n(\xx)=1.$
				
				\item $\lim\limits_{n\to\infty}B_n(\xx)=1.$
				
				\item $\lim\limits_{n\to\infty}G_n(\xx)=G(\xx)$, where $G$ is a continuous non-vanishing function on $\ovl{\Ocal}$.
			\end{enumerate}
		\end{lem}
	\end{proof}
	
	\subsection{Proof of Lemma \ref{all limits}}
	
	In this section, we prove Lemma \ref{all limits}.
	In what follows, we fix a uniformizer $\tt$ of $K$.
	
	\begin{claim} 		
		$\lim\limits_{n\to\infty}F\left(\varphi^{(n-1)}\left(\alpha_n(\xx)\right)\right)=F(\xx).$
	\end{claim}
	
	\begin{proof}
		Note that \[\varphi^{(n-1)}\left(\alpha_n(\xx)\right)\equiv \xx\pmod{\ppi^{n}}.\]
		Thus, we have that $\varphi^{(n-1)}\left(\alpha_n(\xx)\right)\to \xx$ as $n\to \infty$. By the continuity of $F$, the result follows.
	\end{proof}
	
	Since $F$ is continuous and non-vanishing on the compact set $\ovl{\Ocal}$  which is a finite product of the compact sets $\ovl{\Ocal}_\idx$), there exist $\delta_1,\delta_2\in\Z$ such that for all $\xx\in \ovl{\Ocal}$,x
	\begin{equation}\label{uni bounded of F}
		\delta_1\leq\ord_\tt F(\xx)\leq \delta_2.
	\end{equation}
	Moreover, since $F$ is continuous on $\ovl{\Ocal}$, it is uniformly continuous. Hence, given $\rr > 0$, there exists $N_\rr\in\bbN$ such that
	\begin{equation}\label{uni conti of F}
		\xx\equiv \yy\pmod{\ppi^{N_\rr}}\implies F(\xx)\equiv F(\yy)\pmod{\tt^{\rr+\delta_2}}.
	\end{equation}
	
	\begin{claim}    \label{limit of A_n}
		$\lim\limits_{n\to\infty}A_n(\xx)=1.$
	\end{claim}
	
	\begin{proof}
		Note that
		\[\alpha_n(\xx)\equiv\beta_n(\xx)\pmod{\ppi^{2n-2}}.\]
		Thus by $\eqref{uni conti of phi}$, for $i=0,1,\ldots,n-2$, we have
		\[\varphi^{(i)}\left(\alpha_n(\xx)\right)\equiv \varphi^{(i)}\left(\beta_n(\xx)\right)\pmod{\ppi^n}.\]
		So for $n\geq N_\rr$, 
		\[\varphi^{(i)}\left(\alpha_n(\xx)\right)\equiv \varphi^{(i)}\left(\beta_n(\xx)\right)\pmod{\ppi^{N_\rr} }.\]
		Then \eqref{uni conti of F} implies that
		\[F\left(\varphi^{(i)}\left(\alpha_n(\xx)\right)\right)\equiv F\left(\varphi^{(i)}\left(\beta_n(\xx)\right)\right)\pmod{\tt^{\rr+\delta_2}}.\] 
		By \eqref{uni bounded of F} we know $\ord_\tt F(\xx)\leq \delta_2$ for all $\xx \in \ovl{\Ocal}$, which shows that
		\[a_{i,n}(\xx):=\frac{F\left(\varphi^{(i)}\left(\beta_n(\xx)\right)\right)}{F\left(\varphi^{(i)}\left(\alpha_n(\xx)\right)\right)}\equiv 1\pmod{\tt^\rr}.\]
		Therefore, we obtain 
		\[A_n(\xx)=\prod_{i=0}^{n-2}a_{i,n}(\xx)\equiv 1\pmod{\tt^\rr}.\]
		This completes the proof.
	\end{proof}
	
	\begin{claim}
		$\lim\limits_{n\to\infty}B_n(\xx)=1.$
	\end{claim}
	
	\begin{proof}
		Note that
		\[\varphi^{(n)}\left(\alpha_n(\xx)\right)\equiv \varphi^{(n-1)}\left(\beta_n\left(\varphi(\xx)\right)\right)\pmod{\ppi^{2n-2}}. 
		\]
		Thus by \eqref{uni conti of phi}, for $i=0,1,\ldots,n-2$, we have
		\[\varphi^{(n+i)}\left(\alpha_n(\xx)\right)\equiv \varphi^{(n-1+i)}\left(\beta_n\left(\varphi(\xx)\right)\right)\pmod{\ppi^{n}}.\]
		The result now follows from a similar argument as in Claim \ref{limit of A_n}.
	\end{proof}
	
	\begin{claim}   \label{limit of G_n}
		$\lim\limits_{n\to\infty}G_n(\xx)=G(\xx)$ where $G$ is a continuous non-vanishing function on $\ovl{\Ocal}$.
	\end{claim}
	
	\begin{proof}
		We first show that $\{G_n\}_{n=1}^\infty$ is uniformly bounded away from zero. That is, there exist integers $M_1, M_2$ such that for all $n\in\bbN$ and all $\xx\in\ovl{\Ocal}$,
		\[M_1\leq \ord_\tt G_n(\xx)\leq M_2.\]
		Note that by \eqref{bt},
		$$\left|\left(\frac{\bb}{\1-\ppi}\right)_\idx\right| = \left|\frac{b_\idx}{1-\pi_\idx}\right|<1$$
		for each $\idx$, and
		\[\varphi\left(\frac{\bb}{\1-\ppi}\right)=\frac{\bb}{\1-\ppi}. \]
		Thus by \eqref{new condition}, 
		\[F\left(\frac{\bb}{\1-\ppi}\right)=1.\]
		By continuity of $F$, there exists a positive integer $N$ such that for all $\yy\in\ovl{\Ocal}$ with
		\[\yy\equiv \frac{\bb}{\1-\ppi}\pmod{\ppi^{N}},\]  
		we have
		\[F(\yy)\equiv F\left(\frac{\bb}{\1-\ppi}\right)=1\pmod{\tt}. \]
		We note that in this case, $\ord_\tt F(\yy)=0$.
		
		Now, notice that 
		\[\beta_n(\xx)\equiv \frac{\bb}{\1-\ppi}\pmod{\ppi^{n-1}}.\]
		Hence by \eqref{uni conti of phi}, for $i = 0,1,\ldots,n-N-1$,
		\[\varphi^{(i)}\left(\beta_n(\xx)\right)
		\equiv \varphi^{(i)}\left(\frac{\bb}{\1-\ppi}\right)
		= \frac{\bb}{\1-\ppi} \pmod{\ppi^{N}},\]
		which implies that for $i = 0,1,\ldots,n-N-1$,
		\[\ord_\tt F\left(\varphi^{(i)}\left(\beta_n(\xx)\right)\right)=0.\]
		Therefore,
		\[\ord_\tt G_n(\xx)=\ord_\tt\left[\prod_{i=n-N}^{n-2}F\left(\varphi^{(i)}\left(\beta_n(\xx)\right)\right)\right]^{-1}.\]
		From \eqref{uni bounded of F} we know that for all $\xx\in\ovl{\Ocal}$,
		\[\delta_1\leq\ord_\tt F(\xx)\leq \delta_2,\]
		which implies that
		\[-(N-1)\delta_2\leq \ord_\tt G_n(\xx)\leq -(N-1)\delta_1\]
		for all $\xx\in\ovl{\Ocal}$.
		This shows that $\{G_n\}_{n=1}^\infty$ is uniformly bounded away from zero.
		
		Next, we show that $\{G_n\}_{n=1}^\infty$ converges uniformly on $\ovl{\Ocal}$.
		By the completeness of $K$ together with non-Archimedean property, it suffices to show that $\{G_n-G_{n+1}\}_{n=1}^\infty$ converges uniformly to $0$ on $\ovl{\Ocal}$.
		Since $G_{n+1}$ is non-vanishing, we have
		\[G_n-G_{n+1}=G_{n+1}\left(\frac{G_n}{G_{n+1}}-1\right).\]
		And since $G_{n+1}$ is uniformly bounded, it is reduced to showing that given $\rr>0$, there exists a positive integer $N'$ such that for all $n\geq N'$, 
		\[\frac{G_n(\xx)}{G_{n+1}(\xx)}\equiv 1\pmod{\tt^\rr}\]
		for all $\xx\in\ovl{\Ocal}$.
		
		From the definition of $G_n(\xx)$ (see \eqref{def of G_n}), we have
		\[\frac{G_{n}(\xx)}{G_{n+1}(\xx)}=\frac{\prod_{i=0}^{n-1}F\left(\varphi^{(i)}\left(\beta_{n+1}(\xx)\right)\right)}{\prod_{i=0}^{n-2}F\left(\varphi^{(i)}\left(\beta_n(\xx)\right)\right)}=F\left(\beta_{n+1}(\xx)\right)\frac{\prod_{i=0}^{n-2}F\left(\varphi^{(i+1)}\left(\beta_{n+1}(\xx)\right)\right)}{\prod_{i=0}^{n-2}F\left(\varphi^{(i)}\left(\beta_n(\xx)\right)\right)}.\]
		Recall that by \eqref{new condition}, we have $$F\left(\frac{\bb}{\1-\ppi}\right)=1.$$ We have also seen that
		\[\beta_{n+1}(\xx)\equiv \frac{\bb}{\1-\ppi}\pmod{\ppi^{n}}.\]
		By the continuity of $F$, there is a positive integer $N''$ such that for $n\geq N''$, we have
		\[F\left(\beta_{n+1}(\xx)\right)\equiv 1\pmod{\tt^\rr}.\]
		On the other hand, note that 
		\[\varphi\left(\beta_{n+1}(\xx)\right)\equiv\beta_{n}(\xx)\pmod{\ppi^{2n-2}}.\]
		Thus by \eqref{uni conti of phi}, for $i=0,1,\ldots,n-2$, we have
		\[\varphi^{(i+1)}\left(\beta_{n+1}(\xx)\right)\equiv\varphi^{(i)}\left(\beta_{n}(\xx)\right)\pmod{\ppi^{n}}. \]
		So by \eqref{uni conti of F}, for all $n\geq N_\rr$ and all $i=0,1,\ldots,n-2$,
		\[F\left(\varphi^{(i+1)}\left(\beta_{n+1}(\xx)\right)\right)\equiv F\left(\varphi^{(i)}\left(\beta_{n}(\xx)\right)\right)\pmod{\tt^{\rr+\delta_2}}.\]
		Since from \eqref{uni bounded of F}, $\ord_\tt F(\xx)\leq \delta_2$ for all $\xx\in\ovl{\Ocal}$, this implies that for all $n\geq N_\rr$ and all $i=0,1,\ldots,n-2$,
		\[\frac{F\left(\varphi^{(i+1)}\left(\beta_{n+1}(\xx)\right)\right)}{F\left(\varphi^{(i)}\left(\beta_{n}(\xx)\right)\right)}\equiv 1\pmod{\tt^\rr}.\]
		Therefore, for $n\geq \max\{N'',N_\rr\}$, 
		\[\frac{G_n(\xx)}{G_{n+1}(\xx)}\equiv 1\pmod{\tt^\rr}\]
		for all $\xx\in\ovl{\Ocal}$.
		Hence, the desired property is proved.
		
		Finally, one sees from \eqref{def of G_n} that each $G_n$ is continuous, so $\{G_n\}_{n=1}^\infty$ converges to a continuous function $G$. Moreover, since $\{G_n\}_{n=1}^\infty$ is uniformly bounded away from zero, the limit $G$ is non-vanishing.
	\end{proof}
	
	\section{Generalization}
	
	Now, assume $\ovl{\Ocal}$ is a direct product of finitely many rings $\ovl{\Ocal}_\idx$, where each $\ovl{\Ocal}_\idx$ is the ring of integers of some non-Archimedean local field with a fixed uniformizer.
	For each $\idx$, we let $\pi_\idx \in \ovl{\Ocal}_\idx$ be a power of the given uniformizer and $S_\idx \sbe \ovl{\Ocal}_\idx$ be a complete set of representatives of the quotient ring $\ovl{\Ocal}_\idx/(\pi_\idx)$.
	For $\xx = (\xx_t)_t \in\ovl{\Ocal}$, we write
	$$
	\xx_\idx:=\sum_{i=0}^\infty x_{\idx,i}\pi_\idx^i \in \ovl{\Ocal}_\idx,
	$$
	where $x_{\idx,i}\in S_\idx$ for each $\idx$ and $i$.
	We similarly define
	$\varphi:\ovl{\Ocal}\to\ovl{\Ocal}$ by $\varphi(\xx)=(\varphi_\idx(\xx_\idx))_\idx$ where each $\varphi_\idx:\ovl{\Ocal}_\idx\to\ovl{\Ocal}_\idx$ is given by
	$$\varphi_\idx(\xx_\idx):=\sum_{i=0}^\infty x_{\idx,i+1}\pi_\idx^i.$$
	Then Theorem \ref{main result} can be extended to the following.
	
	\begin{thm}
		Given any complete non-Archimedean field $K$, let $\GGamma:\ovl{\Ocal}\to K$ be a continuous non-vanishing function. Then $H:\ovl{\Ocal}\to K$ is a continuous non-vanishing function satisfying for all $n\in\bbN$ and $\xx\in\ovl{\Ocal}$ with $\varphi^{(n)}(-\xx)=-\xx$,
		\[\prod_{j=0}^{n-1} \GGamma\left(-\varphi^{(j)}(-{\xx})\right)
		= \prod_{j=0}^{n-1} H\left(-\varphi^{(j)}(-{\xx})\right)\]
		if and only if
		\[H(\xx)=\GGamma(\xx)\cdot\frac{G(\xx)}{G\left(-\varphi(-\xx)\right)}\]
		on $\ovl{\Ocal}$, where $G:\ovl{\Ocal}\to K$ is any continuous non-vanishing function.
	\end{thm}
	\begin{proof}
		The ``if'' part follows immediately from the assumption that $\varphi^{(n)}(-\xx)=-\xx$ (cf. \eqref{product-of-G}).
		For the ``only if'' part, let $F(\xx) := H(\xx)/\GGamma(\xx)$.
		Then by changing variables (cf. Theorem \ref{thm-only-if} and \ref{change of variables}), it suffices to prove the following statement:
		Let $F:\ovl{\Ocal}\to K$ be a continuous non-vanishing function satisfying for all $n\in\bbN$ and $\xx\in\ovl{\Ocal}$,
		$$
		\varphi^{(n)}(\xx)=\xx \implies \prod_{j=0}^{n-1}F\left(\varphi^{(j)}(\xx)\right)=1.
		$$
		Then there exists a continuous non-vanishing function $G:\ovl{\Ocal}\to K$ such that
		$$F(\xx)=\frac{G(\xx)}{G(\varphi(\xx))}$$
		for all $\xx \in \ovl{\Ocal}$.
		
		Now, we choose any $b_\idx \in S_\idx$ for each $\idx$ and let $\bb:=(b_t)\in\ovl{\oo}$.
		Define $\alpha_n,\beta_n:\ovl{\oo}\to\ovl{\oo}$ as \eqref{alpha_n} and \eqref{beta_n}.
		Then the same algebraic manipulation provides the equality \eqref{FABG}:
		$$F\left(\varphi^{(n-1)}\left(\alpha_n(\xx)\right)\right)=A_n(\xx)\cdot B_n(\xx)\cdot \frac{G_n(\xx)}{G_n\left(\varphi(\xx)\right)}$$
		where $A_n,B_n,G_n$ are defined as \eqref{An Bn} and \eqref{def of G_n}.
		By almost the identical arguments, Lemma \ref{all limits} still holds.
		(The only difference is that, in Claim \ref{limit of G_n}, there is no condition on absolute values in the current situation.)
		Therefore, the result follows in a similar manner.
	\end{proof}
	
	\subsection*{Acknowledgments}
	
	The authors express their sincere gratitude to Chieh-Yu Chang for suggesting this research topic and for facilitating this collaboration.
	They also greatly appreciate his careful reading of the paper and insightful comments, which have significantly contributed to the development of this work.
	Furthermore, the authors extend their special thanks to the National Science and Technology Council for its financial support over the past few years.
	
	\printbibliography
	
\end{document}